\documentclass[11pt]{amsart}
\usepackage{amsmath,amssymb,amsthm}
\usepackage[margin=1.2in]{geometry}
\usepackage[hidelinks]{hyperref}

\newtheorem{theorem}{Theorem}[section]
\newtheorem{lemma}[theorem]{Lemma}
\newtheorem{proposition}[theorem]{Proposition}
\newtheorem{corollary}[theorem]{Corollary}
\newtheorem{conjecture}[theorem]{Conjecture}
\newtheorem{claim}[theorem]{Claim}
\theoremstyle{definition}
\newtheorem{definition}[theorem]{Definition}
\newtheorem{remark}[theorem]{Remark}
\newtheorem{question}[theorem]{Question}

\newcommand{\Sin}[1]{\Sigma^{\mathrm{in}}_{#1}}
\newcommand{\Pin}[1]{\Pi^{\mathrm{in}}_{#1}}
\newcommand{\dSin}[1]{\mathrm{d}\text{-}\Sigma^{\mathrm{in}}_{#1}}
\newcommand{\Sc}[1]{\Sigma^{\mathrm{c}}_{#1}}
\newcommand{\Pc}[1]{\Pi^{\mathrm{c}}_{#1}}
\newcommand{\bSig}[1]{\boldsymbol{\Sigma}^{0}_{#1}}
\newcommand{\bPi}[1]{\boldsymbol{\Pi}^{0}_{#1}}
\newcommand{\A}{\mathcal{A}}
\newcommand{\B}{\mathcal{B}}
\newcommand{\Ccal}{\mathcal{C}}
\newcommand{\T}{\mathcal{T}}
\newcommand{\Scal}{\mathcal{S}}
\newcommand{\ch}{\operatorname{ch}}
\newcommand{\dg}{\operatorname{deg}}
\newcommand{\Lev}{\operatorname{Lev}}
\newcommand{\ran}{\operatorname{ran}}
\newcommand{\dom}{\operatorname{dom}}
\newcommand{\SR}{\operatorname{SR}}
\newcommand{\SSC}{\operatorname{SSC}}
\newcommand{\pSR}{\operatorname{pSR}}
\newcommand{\Mod}{\operatorname{Mod}}
\newcommand{\leT}{\leq_T}

\title[Scott complexity of trees of finite rank]{Scott complexity of trees of finite rank\\ via degrees of categoricity}


\author{Mohammad A. Mahmoud}
\address{Department of Math and Computational Sciences, University of Toronto Mississauga, Mississauga, ON, Canada}
\email{mo.mahmoud@utoronto.ca}

\author{Mostafa Mirabi}
\address{The Taft School, Watertown, CT 06795, USA and Wesleyan University, Middletown, CT 06459, USA}
\email{mmirabi@wesleyan.edu}
\urladdr{https://sites.google.com/site/mostafamirabi}

\begin{document}

\begin{abstract}
In earlier work \cite{Mah19}, the first author constructed, for each finite $m\geq 1$, a computable tree $\A_{m+1}$ of rank $m+1$ whose strong degree of categoricity is $\mathbf{0}^{(2m)}$, and showed this degree is optimal at each rank. Those results are lightface: they concern Turing degrees of isomorphisms between computable copies. In this paper we determine the boldface content of the construction. We isolate a transfer principle: a degree-of-categoricity lower bound that holds uniformly relative to every oracle defeats $L_{\omega_1\omega}$-definability of automorphism orbits outright. We verify that the construction of \cite{Mah19} has this uniformity, and deduce that $\A_{m+1}$ has Scott rank exactly $2m+1$, with Scott sentence complexity one of $\Sin{2m+1}$, $\dSin{2m+1}$, or $\Pin{2m+2}$. For rank $2$ we carry out a complete, computability-free Scott analysis: the orbit of the level-one nodes of infinite degree is $\Pin{2}$- but not $\Sin{2}$-definable, and $\SSC(\A_2)=\Pin{4}$ exactly, witnessed by a computable $\Pc{4}$ Scott sentence. We then prove $\SSC(\A_{m+1})=\Pin{2m+2}$ for all finite $m$: among the three candidates, only $\Pin{2m+2}$ is consistent with the parameterized Scott rank, and a parameter-reservation argument---naming any finite tuple reserves only finitely many top-level subtrees, and the relativized coding survives on the infinitely many spare subtrees---pins $\pSR(\A_{m+1})=2m+1$, selecting that candidate. These results begin a classification of the Scott sentence complexities of trees of finite rank, in analogy with the Gonzalez--Rossegger analysis of linear orders, and connect the $2\alpha$-jump phenomenon in degrees of categoricity to Scott spectral-gap questions for trees.
\end{abstract}

\maketitle

\section{Introduction}

Two traditions measure how hard it is to identify a countable structure up to isomorphism. The first is computational: a computable structure $\A$ has \emph{degree of categoricity} $\mathbf{d}$ if $\mathbf{d}$ is the least Turing degree computing isomorphisms between any two computable copies of $\A$ \cite{FKM10}. The second is definitional: by Scott's theorem every countable structure has a sentence of $L_{\omega_1\omega}$ characterizing it among countable structures, and one asks for the least complexity --- $\Sin{\alpha}$, $\Pin{\alpha}$, or $\dSin{\alpha}$ --- of such a \emph{Scott sentence} \cite{AGHTT21,Mon15}. The bridge between the two is Montalb\'an's robust Scott rank \cite{Mon15}: a structure has a $\Pin{\alpha+1}$ Scott sentence if and only if all of its automorphism orbits are $\Sin{\alpha}$-definable, if and only if it is \emph{boldface} $\mathbf{\Delta}^0_\alpha$-categorical.

In \cite{Mah19}, we proved that for every computable ordinal $\alpha$ there is a computable tree of rank $\alpha+1$ with strong degree of categoricity $\mathbf{0}^{(2\alpha)}$ when $\alpha$ is finite (and $\mathbf{0}^{(2\alpha+1)}$ when $\alpha$ is infinite), that these are the greatest possible degrees of categoricity at each rank, and that the isomorphism problem for computable trees of rank $\alpha$ is $\Pi^0_{2\alpha}$-complete. The present paper answers the natural follow-up question: \emph{what does the hardness construction of \cite{Mah19} say about definability?} While a strong degree of categoricity bounds the Turing complexity of isomorphisms between computable copies, this does not strictly bound the descriptive complexity of the structure's Scott family in  $L_{\omega_1\omega}$. Because Montalbán's equivalence relies on boldface pointclasses, a non-arithmetic Scott family may not be witnessed by the Turing degrees of isomorphisms between any single fixed pair of computable copies (see Remark \ref{rem:gap}).

Our first observation is a transfer principle that converts \emph{relativized} categoricity hardness into outright non-definability.

\begin{proposition}[Transfer principle; Proposition~\ref{prop:transfer}]\label{prop:intro-transfer}
Let $\A$ be a countable structure in a relational language and let $n\geq 1$ be finite. Suppose that for every $Z\subseteq\omega$ there are $Z$-computable copies $\B_Z,\Ccal_Z$ of $\A$ such that no isomorphism $f\colon\B_Z\to\Ccal_Z$ satisfies $f\leT Z^{(n-1)}$. Then some automorphism orbit of $\A$ is not $\Sin{n}$-definable; consequently $\A$ has no $\Pin{n+1}$ Scott sentence and $\SR(\A)\geq n+1$.
\end{proposition}

The proof is a back-and-forth argument: a parameterless $\Sin{n}$ Scott
family, once coded by a real, yields isomorphisms computable in the
$(n-1)$st jump of that real. This is the boldface form of the relative
$\Delta^0_n$-categoricity argument originating with
Ash--Knight--Manasse--Slaman and Chisholm \cite{AKMS89,Chi90}; the point
is the quantifier order: the oracle $Z$ is chosen \emph{after} a putative
Scott family, so no fixed family can stay ahead of all $Z$. Our second observation is that the construction of \cite{Mah19} is uniform in an oracle --- this uniformity was built into \cite{Mah19} (the representation lemmas there are proved for an arbitrary oracle $U$) --- and hence supplies the hypothesis of the transfer principle. Writing $[n:N]$ for the tree types of \cite{Mah19} (Definition~\ref{def:types}) and
\[
\A_{m+1} \;:=\; [m+1:\omega],
\]
the common isomorphism type of the two copies built in \cite{Mah19} at rank $m+1$, we obtain:

\begin{theorem}[Theorems~\ref{thm:relhard} and \ref{thm:SR}]\label{thm:intro-general}
Let $m\geq 1$ be finite.
\begin{enumerate}
\item For every $Z\subseteq\omega$ there are $Z$-computable copies $\T_Z,\Scal_Z$ of $\A_{m+1}$ such that every isomorphism $f\colon \T_Z\to\Scal_Z$ computes $Z^{(2m)}$.
\item $\SR(\A_{m+1}) = 2m+1$. In particular some automorphism orbit of $\A_{m+1}$ is not $\Sin{2m}$-definable, while every orbit is $\Sin{2m+1}$-definable.
\item $\SSC(\A_{m+1}) \in \{\Sin{2m+1},\, \dSin{2m+1},\, \Pin{2m+2}\}$.
\end{enumerate}
\end{theorem}

For the smallest nontrivial case we determine $\SSC(\A_2)$ exactly, by a direct syntactic analysis with no computability anywhere.

\begin{theorem}[Section~\ref{sec:ranktwo}]\label{thm:intro-rank2}
Let $\A_2=[2:\omega]$ be the rank-$2$ tree in which, among the level-one nodes, every finite degree occurs infinitely often and infinite degree occurs infinitely often. Then:
\begin{enumerate}
\item the orbit $O_\omega$ of level-one nodes of infinite degree is $\Pin{2}$-definable but not $\Sin{2}$-definable;
\item $\A_2$ has a $\Pin{4}$ Scott sentence but no $\Pin{3}$, $\Sin{3}$, or $\dSin{3}$ Scott sentence; hence
\[
\SSC(\A_2)=\Pin{4}, \qquad \SR(\A_2)=\pSR(\A_2)=3;
\]
\item the canonical computable copy of $\A_2$ has a c.e.\ Scott family of $\Sc{3}$ formulas and therefore a computable $\Pc{4}$ Scott sentence; the computable Scott sentence complexity of $\A_2$ is also exactly $\Pc{4}$.
\end{enumerate}
\end{theorem}

Theorem~\ref{thm:intro-rank2} aligns perfectly with the lightface picture of \cite{Mah19} at rank $2$: strong degree of categoricity $\mathbf{0}''$ corresponds to boldface $\Delta^0_3$-categoricity and a $\Pin{4}$ Scott sentence; the failure one level down is witnessed in lightface currency by the non-existence of $\mathbf{0}'$-computable isomorphisms and in boldface currency by the non-$\Sin{2}$-definability of the coding orbit $O_\omega$ --- the internal twin of the $\Pi^0_2$-completeness of $\mathrm{Inf}$. We prove that the same exact alignment persists at every finite rank:

\begin{theorem}[Theorem~\ref{thm:main-general}]\label{thm:intro-main}
$\SSC(\A_{m+1})=\Pin{2m+2}$ for every finite $m\geq 1$.
\end{theorem}

The proof, in Section~\ref{sec:general}, does not run the rank-$2$ back-and-forth engine. It instead shows that the relativized coding of Section~\ref{sec:relhard} survives naming any finite tuple of parameters: one reserves the finitely many top-level subtrees that meet the tuple and runs the coding on the infinitely many spare subtrees, so that $\SR(\A_{m+1},\bar a)=2m+1$ for \emph{every} finite $\bar a$ and hence $\pSR(\A_{m+1})=2m+1$, which the $\SSC$--$\pSR$ dictionary of \cite[Table~1]{GR23} converts to $\SSC(\A_{m+1})=\Pin{2m+2}$. Section~\ref{sec:questions} records a separate, more constructive route --- a single back-and-forth substitution lemma (Conjecture~\ref{conj:subst}), whose rank-$2$ instance is exactly the engine of Theorem~\ref{thm:intro-rank2} --- which would re-prove Theorem~\ref{thm:main-general} without computability and localize the critical orbit at $O^{m}_\omega$.

\subsection*{Context and motivation}
Gonzalez and Rossegger \cite{GR23} carried out a comprehensive analysis of the Scott sentence complexities realized by linear orders. The class of trees is a natural next target: like linear orders, trees satisfy Vaught's conjecture (Steel) and are Borel complete but not faithfully Borel complete, and Harrison-Trainor and Kim \cite{HTK26} recently proved that Scott spectral gaps for trees are bounded: every $\Pi_\alpha$-axiomatizable class of trees has a member of Scott rank at most $\alpha+2$, with a lower-bound example at the bottom level ($\Pi_2$ theories all of whose models have Scott rank at least $3$). The trees $\A_{m+1}$, and the calibrated family of types $[k:N]$ underlying them, provide computably presented witnesses pinning the $\Pi$-column of a prospective tree analogue of the Gonzalez--Rossegger table at the even successor levels, with c.e.\ Scott families and exact computable Scott sentences --- the tree counterpart of the effectivity remarks in \cite{GR23}. They also give, for each finite $m$, structures in which a single distinguished orbit concentrates a $\Pin{2m}\setminus\Sin{2m}$ ambiguity; iterating this ambiguity through every level of a model is, we expect, what sharp lower bounds for the Harrison-Trainor--Kim gap must look like (see Section~\ref{sec:questions}). Finally, the relativization-plus-transfer method of Sections~\ref{sec:transfer}--\ref{sec:relhard} is general: any degree-of-categoricity construction that is uniform in an oracle yields boldface non-definability for its isomorphism type, and we state the principle so that it can be quoted.

\subsection*{Organization} Section~\ref{sec:prelim} fixes notation and recalls the Scott analysis toolkit. Section~\ref{sec:transfer} proves the transfer principle. Section~\ref{sec:relhard} relativizes the construction of \cite{Mah19}. Section~\ref{sec:SR} computes $\SR(\A_{m+1})$ and narrows $\SSC(\A_{m+1})$ to three candidates. Section~\ref{sec:ranktwo} contains the complete rank-$2$ analysis. Section~\ref{sec:general} proves $\SSC(\A_{m+1})=\Pin{2m+2}$ in general by parameter reservation. Section~\ref{sec:questions} discusses the substitution-lemma program and collects the open questions.

\section{Preliminaries}\label{sec:prelim}

\subsection{Trees and types}
We follow \cite{Mah19}. A \emph{tree} is a structure $(T,\prec)$ in the relational language $\{\prec\}$ such that $\prec$ is a strict partial order, the set of $\prec$-predecessors of every element is finite and linearly (hence well) ordered, and there is a $\prec$-least element, the \emph{root} $\rho$. All trees in this paper have no infinite paths. $\Lev_T(x)=|\{y: y\prec x\}|$; \emph{rank} is defined by assigning terminal nodes rank $0$ and letting $\mathrm{rk}(x)=\sup\{\mathrm{rk}(y)+1: y \text{ an immediate successor of } x\}$; the rank of $T$ is the rank of $\rho$. For $u\in T$, $T[u]$ is the subtree rooted at $u$, and $\ch(u)$ is the set of immediate successors (\emph{children}) of $u$. In a tree of rank at most $2$, the successors of a level-one node are exactly its children; we write $\dg(u)=|\ch(u)|\in\omega\cup\{\omega\}$ for the \emph{degree} of a level-one node $u$.

A \emph{computable tree} is one whose universe and order are computable; as in \cite{Mah19}, all trees constructed below also have computable successor relation, and the results are insensitive to this distinction. We work in the language $\{\prec\}$ throughout; transfers to the root-and-parent language of \cite{HTK26} are routine for trees of bounded rank but are not carried out here.

\begin{definition}[\cite{Mah19}]\label{def:types}
For finite $n\geq 1$ and $N\in\omega\cup\{\omega\}$, the tree type $[n:N]$ is defined inductively: $[1:N]$ is the tree of rank $\leq 1$ with exactly $N$ children of the root; $[n:N]$ is the tree of rank $n$ in which exactly $N$ of the level-one nodes are roots of subtrees of type $[n-1:\omega]$ and, for every $k\in\omega$ (including $k=0$), infinitely many level-one nodes are roots of subtrees of type $[n-1:k]$.
\end{definition}

We set $\A_{m+1}:=[m+1:\omega]$ for finite $m\geq 1$. Thus $\A_2$ is the rank-$2$ tree whose level-one nodes realize every degree in $\omega\cup\{\omega\}$, each infinitely often. The two computable copies constructed in \cite{Mah19} at rank $m+1$ (Proposition~2.3 there for $m=1$; Theorem~3.3 in general) are both of type $[m+1:\omega]$: in the coding copy the even level-one nodes $(2x)$ root subtrees of type $[m:\omega]$ exactly when $x\notin\emptyset^{(2m)}$, and $\emptyset^{(2m)}$ is coinfinite, while the odd nodes supply each finite type infinitely often.

\subsection{Infinitary logic and Scott analysis}
The \emph{Scott sentence complexity} $\SSC(\A)$ is the least of the complexities $\Sin{\alpha},\Pin{\alpha},\dSin{\alpha}$ (under the natural inclusion ordering, in which $\Sin{\alpha},\Pin{\alpha}<\dSin{\alpha}<\Sin{\alpha+1},\Pin{\alpha+1}$) of a Scott sentence of $\A$; this is well defined by \cite{AGHTT21}. The (parameterless) \emph{Scott rank} $\SR(\A)$ is the least $\alpha$ such that every automorphism orbit of every finite tuple of $\A$ is $\Sin{\alpha}$-definable without parameters; the parameterized rank $\pSR(\A)$ is the least $\SR((\A,\bar p))$ over finite parameter tuples $\bar p$. We use freely the relations between $\SSC$, $\SR$, $\pSR$ tabulated in \cite[Table~1]{GR23} (from Montalb\'an's book): in particular $\SSC(\A)=\Pin{\alpha+1}$ implies $\pSR(\A)=\SR(\A)=\alpha$.

The fundamental correspondence is:

\begin{theorem}[Montalb\'an \cite{Mon15}]\label{thm:montalban}
Let $\A$ be a countable structure and $\alpha$ a countable ordinal. The following are equivalent:
\begin{enumerate}
\item $\A$ has a $\Pin{\alpha+1}$ Scott sentence;
\item every automorphism orbit of every finite tuple of $\A$ is $\Sin{\alpha}$-definable without parameters, i.e.\ $\A$ has a parameterless Scott family of $\Sin{\alpha}$ formulas;
\item the set $\{\B\in\Mod(L): \B\cong\A\}$ is $\bPi{\alpha+1}$;
\item $\A$ is \emph{uniformly boldface} $\mathbf{\Delta}^0_\alpha$-categorical.
\end{enumerate}
Consequently $\SR(\A)$ equals the least $\alpha$ such that $\A$ has a $\Pin{\alpha+1}$ Scott sentence.
\end{theorem}

In (3) and (4) the pointclasses are boldface: a real parameter is allowed. For finite $\alpha$, boldface $\mathbf{\Delta}^0_\alpha$-categoricity means that for all copies $X,Y$ of $\A$ there is an isomorphism computable in $(X\oplus Y\oplus p)^{(\alpha-1)}$ for some fixed real $p$. This parameter is not decorative; it is the reason Section~\ref{sec:transfer} is needed (Remark~\ref{rem:gap}).

We also use the following theorem of A.~Miller \cite{Mil83}; the short proof we follow is due to Alvir \cite[Theorem~2.5]{AGHTT21}.

\begin{theorem}[A.~Miller; see {\cite[Theorem~2.5]{AGHTT21}}]\label{thm:miller}
Let $\A$ be a countable structure and $\alpha$ a countable ordinal. If $\A$ has both a $\Sin{\alpha}$ Scott sentence and a $\Pin{\alpha}$ Scott sentence, then $\A$ has a $\dSin{\beta}$ Scott sentence for some $\beta<\alpha$. Moreover, if $\alpha=\gamma+1$ is a successor, then $\A$ has a $\dSin{\gamma}$ Scott sentence.
\end{theorem}

Finally, we recall the evaluation bound for infinitary formulas in copies \cite[Ch.~7]{AK00}, relativized: if $\varphi$ is an $X_0$-computable $\Pin{\beta}$ formula ($\beta\geq 1$ finite) and $\B$ is a copy of a structure with universe $\omega$, then $\{\bar b: \B\models\varphi(\bar b)\}$ is $\Pi^0_\beta(\B\oplus X_0)$, uniformly; in particular it is decidable in $(\B\oplus X_0)^{(\beta)}$.

\section{A transfer principle: from relativized categoricity to non-definability}\label{sec:transfer}

\begin{remark}\label{rem:gap}
Before proving the principle we record why it is needed, i.e.\ why a lightface categoricity lower bound does not by itself defeat definability. Suppose every orbit of $\A$ is $\Sin{n}$-definable, by formulas jointly coded by a real $X_0$. The argument below produces, between any computable copies $\B,\Ccal$, an isomorphism computable in $X_0^{(n-1)}$. If $X_0\geq_T\emptyset^{(2m)}$ --- and nothing bounds the complexity of an abstract Scott family --- this is perfectly compatible with the statement ``every isomorphism between $\B$ and $\Ccal$ computes $\emptyset^{(2m)}$,'' for any $n$ whatsoever. A fixed pair of computable copies can never see a sufficiently non-arithmetic Scott family. What no countable family can survive is hardness that relativizes: the oracle in Proposition~\ref{prop:transfer} is chosen after the family.
\end{remark}

\begin{proposition}[Transfer principle]\label{prop:transfer}
Let $\A$ be a countable structure in a relational language and let $n\geq 1$ be finite. Suppose that for every $Z\subseteq\omega$ there exist $Z$-computable copies $\B_Z,\Ccal_Z$ of $\A$ such that no isomorphism $f\colon\B_Z\to\Ccal_Z$ satisfies $f\leT Z^{(n-1)}$. Then some automorphism orbit of a finite tuple of $\A$ is not $\Sin{n}$-definable. Consequently (by Theorem~\ref{thm:montalban}) $\A$ has no $\Pin{n+1}$ Scott sentence, and $\SR(\A)\geq n+1$.
\end{proposition}

\begin{proof}
Suppose toward a contradiction that every orbit is $\Sin{n}$-definable, and choose for each orbit $O$ of a finite tuple a parameterless defining formula $\varphi_O\in\Sin{n}$, in normal form
\[
\varphi_O \;=\; \bigvee_{i\in\omega} \exists\bar y_i\, \psi_{O,i}(\bar x,\bar y_i), \qquad \psi_{O,i}\in\Pin{n-1}.
\]
There are countably many orbits; fix a single real $X_0$ coding an enumeration $(\varphi_k)_{k\in\omega}$ of the chosen formulas together with their syntax, so that each $\varphi_k$ is an $X_0$-computable $\Sigma_n$ formula and each matrix $\psi_{k,i}$ an $X_0$-computable $\Pi_{n-1}$ formula in the sense of \cite[Ch.~7]{AK00} relativized to $X_0$. Two properties of the list are used below: every $\varphi_k$ defines a single automorphism orbit of $\A$, and every orbit is defined by some $\varphi_k$.

\begin{claim}\label{claim:bf}
For any two copies $\B,\Ccal$ of $\A$ with universe $\omega$ there is an isomorphism $f\colon \B\to\Ccal$ with $f\leT (\B\oplus\Ccal\oplus X_0)^{(n-1)}$.
\end{claim}

\begin{proof}[Proof of Claim]
We build $f$ as the union of a chain of finite partial injections $\bar a\mapsto\bar b$ maintaining the invariant
\[
(\ast)\qquad \text{for all } k: \quad \B\models\varphi_k(\bar a) \iff \Ccal\models\varphi_k(\bar b).
\]
Since each $\varphi_k$ defines an orbit of $\A$, and definable sets are carried to each other by any isomorphisms $\B\cong\A\cong\Ccal$, $(\ast)$ says exactly that $\bar a$ and $\bar b$ lie in corresponding orbits; in particular tuples satisfying $(\ast)$ realize the same atomic formulas, so the union of the chain, if total and onto in both directions, is an isomorphism.

Forward step (the back step is symmetric). Given $\bar a\mapsto\bar b$ satisfying $(\ast)$ and a new element $a$ of $\B$, search over tuples $(k,i,\bar w,b,i',\bar w')$ for one with
\[
\B\models\psi_{k,i}(\bar a a,\bar w) \quad\text{and}\quad \Ccal\models\psi_{k,i'}(\bar b b,\bar w').
\]
Each condition is a $\Pi^0_{n-1}(\B\oplus\Ccal\oplus X_0)$ fact about an $X_0$-computable $\Pi_{n-1}$ formula, hence decidable in $(\B\oplus\Ccal\oplus X_0)^{(n-1)}$, uniformly. (For $n=1$ the matrices are quantifier-free conjunctions and the conditions are decidable in $\B\oplus\Ccal\oplus X_0$ directly; the count is the same.) If the search succeeds, then $\B\models\varphi_k(\bar a a)$ and $\Ccal\models\varphi_k(\bar b b)$; since $\varphi_k$ defines a single orbit, $\bar a a$ and $\bar b b$ lie in corresponding orbits, so $(\ast)$ holds of the extended map and we set $f(a)=b$.

The search terminates. Indeed, let $\varphi_k$ be the formula chosen for the orbit of (the image in $\A$ of) $\bar a a$; then $\B\models\varphi_k(\bar a a)$, providing $i$ and $\bar w$. For the $\Ccal$ side, since $\bar b$ corresponds to $\bar a$, there is an automorphism of $\A$ matching them up; composing, some $b$ in $\Ccal$ puts $\bar b b$ in the orbit corresponding to that of $\bar a a$, whence $\Ccal\models\varphi_k(\bar b b)$, providing $i'$ and $\bar w'$.

The whole construction is an unbounded search with a $(\B\oplus\Ccal\oplus X_0)^{(n-1)}$-decidable test at each step, terminating at every step; hence $f\leT(\B\oplus\Ccal\oplus X_0)^{(n-1)}$. Note carefully that we never decide $\Sin{n}$ truth (which would cost the $n$-th jump): to extend the map we only \emph{find witnesses}, verifying $\Pi_{n-1}$ matrices on candidates. This off-by-one is what makes the contradiction below possible.
\end{proof}

Now take $Z:=X_0$ and apply the hypothesis to obtain $Z$-computable copies $\B_Z,\Ccal_Z$ of $\A$ with no $Z^{(n-1)}$-computable isomorphism between them. By the Claim there is an isomorphism $f\leT(\B_Z\oplus\Ccal_Z\oplus X_0)^{(n-1)}\leT Z^{(n-1)}$, a contradiction.
\end{proof}

\begin{remark}
Proposition~\ref{prop:transfer} is stated for finite $n$ only; at infinite levels the correspondence between $\Delta^0_\alpha$ and iterated jumps shifts by one, exactly mirroring the successor/limit case split in \cite{Mah19}. We do not pursue infinite ranks here; see Section~\ref{sec:questions}.
\end{remark}

\section{Relativizing the categoricity construction}\label{sec:relhard}

\begin{theorem}\label{thm:relhard}
Let $m\geq 1$ be finite. For every $Z\subseteq\omega$ there exist trees $\T_Z$ and $\Scal_Z$ such that:
\begin{enumerate}
\item $\T_Z$ and $\Scal_Z$ are $Z$-computable (universe, order, and successor relation);
\item $\T_Z\cong\Scal_Z\cong\A_{m+1}$ --- the same isomorphism type for every $Z$;
\item every isomorphism $f\colon\T_Z\to\Scal_Z$ satisfies $Z^{(2m)}\leT f$.
\end{enumerate}
\end{theorem}

\begin{proof}
We first treat $m=1$ in full, since it displays the mechanism, and then explain why the general case is the construction of \cite[Thm.~3.3]{Mah19} relativized verbatim.

\emph{Case $m=1$.} Fix a $Z$-computable enumeration $(W^Z_{i,s})$ of the $Z$-c.e.\ sets. Let $\T_Z$ have root $()$, level-one nodes $(i)$ for $i\in\omega$, and level-two nodes $(i,s)$ exactly when $W^Z_{i,s+1}\neq W^Z_{i,s}$. Let $\Scal_Z$ have root $()$, level-one nodes $(i)$, and level-two nodes $(2i,j)$ for all $i,j$, together with $(2\langle i,k\rangle+1,j)$ for $j<i$. Both trees are $Z$-computable with $Z$-computable successor relation. In $\T_Z$, the node $(i)$ has degree equal to the number of enumeration stages of $W^Z_i$, which is finite iff $W^Z_i$ is finite; by the relativized padding lemma and the existence of $Z$-c.e.\ sets of every size, every degree in $\omega\cup\{\omega\}$ occurs infinitely often. The same holds in $\Scal_Z$ by inspection. Hence $\T_Z\cong\Scal_Z\cong[2:\omega]=\A_2$, and the type does not depend on $Z$. If $f\colon\T_Z\to\Scal_Z$ is an isomorphism, then for every $r$,
\[
r\in\mathrm{Inf}^Z \iff \dg((r))=\omega \iff \dg(f((r)))=\omega \iff f((r)) \text{ is even},
\]
a condition computable from $f$ alone. Since $\mathrm{Inf}^Z$ is $\Pi^0_2(Z)$-complete under $\leq_1$ and $\overline{Z''}\in\Pi^0_2(Z)$, we get $Z''\leT \mathrm{Inf}^Z\leT f$.

\emph{Case $m\geq 2$.} We claim that the proof of \cite[Thm.~3.3]{Mah19} relativizes to an arbitrary oracle $Z$, with the same isomorphism type as output. The relevant inventory:
\begin{enumerate}
\item[(a)] The representation results \cite[Lem.~2.4]{Mah19} are stated there for an arbitrary oracle $U$, with the representing functions \emph{computable} (not merely $U$-computable) via the relativized $s$-$m$-$n$ theorem; this strong uniformity is the load-bearing point.
\item[(b)] The structural representation \cite[Lem.~3.1]{Mah19} is likewise stated for an arbitrary $U$; its movable-marker construction defines $U$-c.e.\ sets uniformly in all indices and, as noted there, is uniform in $U$: the marker $r^U(x,y,s)$ is computed from the stage-$s$ approximations $W^U_{\cdot,s}$, and the verification (both cases) nowhere uses $U=\emptyset$.
\item[(c)] Consequently \cite[Lem.~3.2]{Mah19} relativizes: for every $D\in\Sigma^Z_{2k+2}$ there is a uniformly $Z$-computable sequence $(\T^x)_{x\in\omega}$ of trees of rank $k+1$ with $\T^x\cong[k+1:\omega]$ iff $x\notin D$ and $\T^x\cong[k+1:j]$ for some finite $j$ iff $x\in D$. The induction is identical; the base case puts $(x,s)$ at level one iff $W^Z_{g(x),s+1}\neq W^Z_{g(x),s}$, a $Z$-computable condition; and crucially the tree \emph{types} $[k:N]$ are oracle-free isomorphism types.
\end{enumerate}
Now take $D:=Z^{(2m)}$, which is $\Sigma^Z_{2m}=\Sigma^Z_{2(m-1)+2}$, and let $(\T^x)$ be as in (c) with $k=m-1$, so each $\T^x$ has rank $m$ and type $[m:\omega]$ iff $x\notin Z^{(2m)}$. Let $\T_Z$ have root $()$, with each even level-one node $(2x)$ the root of a copy of $\T^x$ and the odd level-one nodes roots of infinitely many copies of $[m:k]$ for every finite $k$; let $\Scal_Z$ be the transparent copy in which every even level-one node roots a $Z$-computable copy of $[m:\omega]$ and the odd nodes again supply all finite types. Both are $Z$-computable. Since $Z^{(2m)}$ is coinfinite --- a cofinite set is computable, while $Z^{(2m)}\geq_T Z'$ is not computable --- infinitely many $x$ lie outside $Z^{(2m)}$, so $\T_Z$ has infinitely many level-one subtrees of type $[m:\omega]$ and infinitely many of each finite type; thus $\T_Z\cong\Scal_Z\cong[m+1:\omega]=\A_{m+1}$, independently of $Z$. Finally, for any isomorphism $f\colon\T_Z\to\Scal_Z$ and every $x$:
\[
x\in Z^{(2m)} \iff \T_Z[(2x)]\not\cong[m:\omega] \iff \Scal_Z[f((2x))]\not\cong[m:\omega] \iff f((2x)) \text{ is odd},
\]
again a parity check computable from $f$ alone; hence $Z^{(2m)}\leT f$.
\end{proof}

\begin{remark}\label{rem:coding-orbit}
Theorem~\ref{thm:relhard} also identifies \emph{where} the coding lives, in both copies: all coded information rides on membership of single level-one nodes in the set
\[
O^{m}_\omega := \{u: \Lev(u)=1 \text{ and } T[u]\cong[m:\omega]\},
\]
which is a single automorphism orbit of $\A_{m+1}$ (for $m=1$ this is part of Lemma~\ref{lem:datum} below; in general it follows from the analogous homogeneity of the types $[m:N]$). In the coding copy, membership of the designated nodes in $O^{m}_\omega$ is a fixed $\Pi^0_{2m}(Z)$-complete set: for $m=1$ this is $\mathrm{Inf}^Z$ (the node $(i)$ lies in $O^1_\omega$ iff $W^Z_i$ is infinite), and for $m\geq 2$ it is the complement of $Z^{(2m)}$ (the node $(2x)$ lies in $O^{m}_\omega$ iff $\T^x\cong[m:\omega]$ iff $x\notin Z^{(2m)}$). These two sets are not equal, but both are $\Pi^0_{2m}(Z)$-complete under $\leq_1$, hence Turing-equivalent to $Z^{(2m)}$, which is all the decoding needs. In the transparent copy, membership in $O^{m}_\omega$ is decided by parity of the level-one node, a $Z$-computable condition. An isomorphism thus transports orbit membership from where it is $\Pi^0_{2m}(Z)$-complete to where it is $Z$-computable; that transport is the entire decoding step.
\end{remark}

\section{The Scott rank of $\A_{m+1}$}\label{sec:SR}

\begin{lemma}\label{lem:relcat}
Let $m\geq 1$ be finite and let $X,Y$ be arbitrary copies (with universe $\omega$ and no computabilty restriction) of a tree of rank at most $m+1$. Then there is an isomorphism $f\colon X\to Y$ with $f\leT (X\oplus Y)^{(2m)}$, uniformly.  In particular every tree of rank at most $m+1$ is (boldface) $\mathbf{\Delta}^0_{2m+1}$-categorical, uniformly and with no parameter.
\end{lemma}

\begin{proof}
This is \cite[Lem.~2.1 and Prop.~2.2]{Mah19} relativized. The proof of \cite[Lem.~2.1]{Mah19} that isomorphism of trees of rank $\alpha$ is (uniformly) $\Pi_{2\alpha}$ is, for finite $\alpha$, a purely syntactic induction with no limit stages, and relativizes verbatim: deciding whether $u$ is a level-one node of $X$ is $\Delta^0_2(X)$, and the displayed count-matching condition has two more quantifier blocks than the isomorphism queries for the level-one subtrees, which have rank $\leq m$. Hence ``$X[u]\cong Y[v]$'' is $\Pi^0_{2m}(X\oplus Y)$ uniformly in $u,v$, and is in particular decidable in $(X\oplus Y)^{(2m)}$. The back-and-forth of \cite[Prop.~2.2]{Mah19} then runs with oracle $(X\oplus Y)^{(2m)}$.
\end{proof}

\begin{theorem}\label{thm:SR}
For every finite $m\geq 1$: $\SR(\A_{m+1})=2m+1$. Hence some automorphism orbit of $\A_{m+1}$ is not $\Sin{2m}$-definable, every orbit is $\Sin{2m+1}$-definable, $\A_{m+1}$ has a $\Pin{2m+2}$ Scott sentence and no $\Pin{2m+1}$ Scott sentence.
\end{theorem}

\begin{proof}
Lower bound: by Theorem~\ref{thm:relhard}, for every $Z$ there are $Z$-computable copies between which every isomorphism computes $Z^{(2m)}$; since $Z^{(2m)}\not\leT Z^{(2m-1)}$, no isomorphism between them is computable in $Z^{(2m-1)}$. Proposition~\ref{prop:transfer} with $n=2m$ yields an orbit that is not $\Sin{2m}$-definable, hence no $\Pin{2m+1}$ Scott sentence and $\SR(\A_{m+1})\geq 2m+1$.

Upper bound: by Lemma~\ref{lem:relcat}, $\A_{m+1}$ is boldface $\mathbf{\Delta}^0_{2m+1}$-categorical; by Theorem~\ref{thm:montalban} ((4)$\Rightarrow$(1),(2)) it has a $\Pin{2m+2}$ Scott sentence and a Scott family of $\Sin{2m+1}$ formulas, so $\SR(\A_{m+1})\leq 2m+1$.
\end{proof}

\begin{corollary}\label{cor:dichotomy}
For every finite $m\geq 1$: $\SSC(\A_{m+1})\in\{\Sin{2m+1},\,\dSin{2m+1},\,\Pin{2m+2}\}$.
\end{corollary}

\begin{proof}
By Theorem~\ref{thm:SR}, $\A_{m+1}$ has a $\Pin{2m+2}$ Scott sentence, so $\SSC(\A_{m+1})$ is one of $\Sin{\beta},\Pin{\beta},\dSin{\beta}$ with class included in $\Pin{2m+2}$. Any candidate whose class is included in $\Pin{2m+1}$ is realized by a $\Pin{2m+1}$ Scott sentence, contradicting Theorem~\ref{thm:SR}. The only candidates included in $\Pin{2m+2}$ but not in $\Pin{2m+1}$ are $\Sin{2m+1}$, $\dSin{2m+1}$, and $\Pin{2m+2}$ (note $\Sin{2m+2}\not\subseteq\Pin{2m+2}$, so it is not a candidate).
\end{proof}
We will show in the next section that for $m=1$ the answer is $\Pin{4}$, and we prove this pattern persists at every finite rank (Theorem~\ref{thm:main-general}).

\section{The rank-two tree: exact Scott sentence complexity}\label{sec:ranktwo}

Throughout this section $\A:=\A_2=[2:\omega]$, with root $\rho$; its level-one nodes realize every degree in $\omega\cup\{\omega\}$, each infinitely often. By the $r=1$ case of Lemma~\ref{lem:datum}, the orbits of single elements are: $\{\rho\}$; the sets $O_k$ ($k\in\omega$) and $O_\omega$ of level-one nodes of degree $k$ resp.\ infinite degree; and the sets $C_\kappa$ ($\kappa\in(\omega\setminus\{0\})\cup\{\omega\}$) of level-two nodes whose parent has degree $\kappa$. All results in this section are purely syntactic; no computability theory is used until Proposition~\ref{prop:computable-ss}.

\subsection{Building blocks}\label{ssec:blocks}
Over trees of rank at most $2$ in the language $\{\prec\}$ we use:
\begin{align*}
\mathrm{root}(x) &:= \forall y\,\neg(y\prec x), && \Pi_{1};\\
\mathrm{lev}_1(x) &:= \exists y\,(y\prec x)\;\wedge\;\forall y\forall z\,\neg(y\prec z\prec x), && \Sigma_1\wedge\Pi_1;\\
\mathrm{lev}_2(x) &:= \exists z\,(z\prec x\wedge \exists y\,(y\prec z)), && \Sigma_1;\\
\dg_{\geq k}(x) &:= \exists y_1\cdots\exists y_k\Big(\bigwedge_{i<j} y_i\neq y_j\wedge\bigwedge_i x\prec y_i\Big), && \Sigma_1;\\
\dg_{=k}(x) &:= \dg_{\geq k}(x)\wedge\neg\dg_{\geq k+1}(x), && \Sigma_{1}\wedge\Pi_1;\\
\dg_{=\omega}(x) &:= \textstyle\bigwedge_{k\in\omega}\dg_{\geq k}(x), && \Pin{2};\\
\mathrm{samepar}(x,y) &:= \exists w\,(w\prec x\wedge w\prec y\wedge \exists v\,(v\prec w)), && \Sigma_1;\\
\mathrm{pardeg}_{=k}(x) &:= \exists w\,(w\prec x\wedge \exists v(v\prec w)\wedge \dg_{=k}(w)), && \Sigma_{2};\\
\mathrm{pardeg}_{=\omega}(x) &:= \exists w\,(w\prec x\wedge \exists v(v\prec w)\wedge \dg_{=\omega}(w)), && \Sin{3}.
\end{align*}

The complexity annotations and intended meanings (degree of a level-one node;
shared level-one parent; parent's degree) follow from $\A$ being a tree of
height two --- every node has level $0$, $1$, or $2$, so the successors of a
level-one node are exactly its children. Conjoining the level clauses to fix
which nodes the degree formulas apply to: $O_k$ is d-$\Sigma_{1}$-definable (by
$\mathrm{lev}_1\wedge\dg_{=k}$), $O_\omega$ is $\Pin{2}$-definable (by
$\mathrm{lev}_1\wedge\dg_{=\omega}$), $C_k$ is $\Sigma_{2}$-definable (by
$\mathrm{pardeg}_{=k}$), and $C_\omega$ is $\Sin{3}$-definable (by
$\mathrm{pardeg}_{=\omega}$); the latter two need no level clause, since the
witness $w$ in $\mathrm{pardeg}$ is forced to be a level-one parent.

\subsection{The extension lemma}
The following finite extension property of rank-$2$ trees with rich level-one spectrum powers all the back-and-forth arguments of this section. All maps below are required to be injective, to preserve and reflect $\prec$, to send root to root, and to preserve levels; we call such a map a \emph{leveled embedding} of its (finite) domain.

\begin{lemma}[Extension Lemma]\label{lem:ext}
Let $A$ and $B$ be trees of rank at most $2$ such that for every $d\in\omega$, $B$ has infinitely many level-one nodes of degree at least $d$. Let $E\subseteq A$ be finite, let $X$ be a subset of $E$ which contains the root of $A$ ($\rho_A\in X$), and let $p\colon X\to B$ be a leveled embedding such that:
\begin{enumerate}
\item[(i)] every level-two element of $X$ whose parent lies in $E$ has its parent in $X$;
\item[(ii)] for every level-one $w\in X$: $\dg_B(p(w))\geq |E\cap\ch_A(w)|$.
\end{enumerate}
Then $p$ extends to a leveled embedding $q\colon E\to B$.
\end{lemma}

\begin{proof}
Call a level-one node $z$ of $B$ \emph{fresh} (at a given stage) if $z\notin\ran(q)$ and $z$ is not the parent of any level-two node in $\ran(q)$; call a child of a node fresh if it is not in $\ran(q)$. At every stage only finitely many nodes are excluded, so fresh level-one nodes of degree $\geq d$ exist for every $d$, by hypothesis on $B$.

Process the elements of $E\setminus X$ in the order: first all level-one elements, then all level-two elements.

\emph{Level-one $e$.} Set $q(e):=$ a fresh level-one node of $B$ of degree $\geq |E\cap\ch_A(e)|$. We check that the extended map is a leveled embedding by verifying that the relation $\prec$ is perfectly preserved between the node we just added to our map ($e$) and every node we have already mapped ($x$). If $x=\rho_A$ (trivial): $\rho_A\prec e$ and $q(x)=\rho_B\prec q(e)$. If $x$ is level-one (also trivial): $e,x$ are incomparable, as are the distinct level-one nodes $q(e),q(x)$. If $x$ is level-two: $x\notin\ch_A(e)$ --- for $x\in X$ this is hypothesis (i) (its parent would have to lie in $X$, but $e\notin X$), and level-two elements of $E\setminus X$ have not been processed yet --- so $e,x$ are incomparable; on the image side, $q(x)$ is not a child of $q(e)$ because $q(e)$ is fresh (not the parent of anything in $\ran(q)$), so $q(e),q(x)$ are incomparable.

\emph{Level-two $e$.} Let $w$ be the parent of $e$ in $A$. If $w\in E$ then $w\in\dom(q)$ (all level-one elements of $E$ are now placed, and $w\in X$ is allowed); set $q(e):=$ a fresh child of $q(w)$. Such a child exists: the number of children of $q(w)$ already used equals the number of previously placed elements of $E\cap\ch_A(w)$, which is below $|E\cap\ch_A(w)|\leq\dg_B(q(w))$ --- by hypothesis (ii) when $w\in X$, and by the choice of $q(w)$ when $w\in E\setminus X$. If $w\notin E$: set $q(e):=$ a fresh child of a fresh level-one node of $B$ of degree $\geq 1$. Checking the new pairs $(e,x)$: if $x=\rho_A$, fine. If $x$ is level-one: $e$ is comparable to $x$ iff $x=w$, and $q(e)$ is comparable to $q(x)$ iff $q(x)$ is the parent of $q(e)$, which by construction and injectivity happens iff $x=w$ (when $w\notin E$, the designated parent of $q(e)$ is fresh, hence not in $\ran(q)$). If $x$ is level-two: $e,x$ are incomparable, and so are any two level-two nodes of $B$. Injectivity is maintained throughout by freshness. (Note that the quantifier-free type does not track \emph{unnamed} shared parents: two level-two elements of $E$ with the same parent outside $E$ may be sent to children of different nodes; this is permitted, since a leveled embedding is only required to preserve and reflect $\prec$ and equality on its domain.)
\end{proof}

\subsection{Orbits and a $\Sin{3}$ Scott family}

\begin{lemma}[Orbit datum]\label{lem:datum}
Tuples (finite) $\bar a,\bar b$ from $\A$ lie in the same automorphism orbit if and only if they have the same \emph{datum}:
\begin{enumerate}
\item[(D1)] the same induced partial-order diagram over the root (including equalities and levels);
\item[(D2)] $\dg(a_i)=\dg(b_i)$ for every level-one entry;
\item[(D3)] the parent structures for all level-two nodes must be identical. If two nodes share a parent on the $\bar{a}$ side, their counterparts must share a parent on the $\bar{b}$ side---even if that parent isn't actually named in the tuple. Furthermore, those unnamed ``ghost'' parents must have the exact same total number of children.
\end{enumerate}
\end{lemma}

\begin{proof}
Invariance of the datum under automorphisms is immediate, so suppose the data of $\bar a$ and $\bar b$ agree; we build an automorphism $g$ with $g(\bar a)=\bar b$.

\emph{Named points.} Set $g(\rho)=\rho$ and $g(a_i)=b_i$. By (D1) this is well defined, injective, and a leveled partial isomorphism on $\{\rho\}\cup\bar a$ (equalities, levels, and $\prec$ among named entries all agree).

\emph{Unnamed parents.} Among the level-two entries of $\bar a$ with unnamed parents, having a common parent is an equivalence relation on their indices; call its blocks the shared-parent classes. For each such class $C$, let $w_C$ be the common parent on the $\bar a$ side; by (D3) the entries $b_i$ ($i\in C$) likewise share a parent $v_C$, and $\dg(w_C)=\dg(v_C)$. Set $g(w_C)=v_C$. Distinct classes have distinct parents on each side, and each $w_C$, $v_C$ is level-one and distinct from every named entry, so $g$ remains a leveled embedding: the only new $\prec$-relations are $\rho\prec w_C$ and $w_C\prec a_i$ ($i\in C$), matched by $\rho\prec v_C$ and $v_C\prec b_i$, reflection holding because an unnamed parent is not a named entry. 

\emph{Children of matched level-one nodes.} Every level-one node now in $\dom(g)$---the named level-one entries (matched by (D2)) and the $w_C$ (matched by (D3))---is sent to one of equal degree carrying equally many named children. Extend $g$ by a bijection between the unnamed children of each matched pair: the counts agree, being equal degrees minus equal finite named counts ($\omega$ minus finite is $\omega$).

\emph{Everything else.} For each $t\in\omega\cup\{\omega\}$, $\A$ has infinitely many level-one nodes of degree $t$, only finitely many of which lie in $\dom(g)$ or $\ran(g)$; match the remainder by a degree-preserving bijection and extend over their children as above. The resulting $g$ is a level-preserving bijection of $\A$ preserving and reflecting $\prec$---an automorphism---with $g(\bar a)=\bar b$.
\end{proof}

\begin{proposition}\label{prop:scottfamily}
Every automorphism orbit of $\A$ is definable by a finite conjunction of
instances of the formulas of \S\ref{ssec:blocks} together with atomic
$\prec$-formulas and (in)equalities; in particular every orbit is
$\Sin{3}$-definable, and $\A$ has a parameterless Scott family of $\Sin{3}$
formulas. Hence $\A$ has a $\Pin{4}$ Scott sentence and $\SR(\A)\leq 3$.
\end{proposition}

\begin{proof}
Given $\bar a$, let $\varphi_{\bar a}(\bar x)$ be the conjunction of the
following clauses:
\begin{itemize}
\item the complete quantifier-free $\{\prec\}$-diagram of $\bar a$
  (including $\mathrm{root}(x_i)$ for the root entry, if the root is in the tuple $\bar a$);
\item $\mathrm{lev}_1(x_i)$ or $\mathrm{lev}_2(x_i)$ for each entry, matching
  its level;
\item $\dg_{=k}(x_i)$ or $\dg_{=\omega}(x_i)$ for each level-one entry,
  matching (D2);
\item for the level-two entries with unnamed parents:
  $\mathrm{samepar}(x_i,x_j)$ or $\neg\mathrm{samepar}(x_i,x_j)$ on each pair,
  and $\mathrm{pardeg}_{=k}(x_i)$ or $\mathrm{pardeg}_{=\omega}(x_i)$, matching
  (D3).
\end{itemize}

Suppose $\A\models\varphi_{\bar a}(\bar b)$. Then $\bar b$ has the same
quantifier-free diagram and levels as $\bar a$ (so (D1) holds), the same
level-one degrees (D2), the same named-parent pattern, the same
shared-parent partition among entries with unnamed parents, and the same
unnamed-parent degrees (D3). For the shared-parent partition note that the
inner clause $\exists v(v\prec w)$ of $\mathrm{samepar}$ forces the common
predecessor $w$ to lie at level one rather than at the root, so
$\mathrm{samepar}$ expresses precisely ``shared level-one parent.'' Thus
$\bar b$ realizes the full datum of $\bar a$, and Lemma~\ref{lem:datum} gives
$\A\models\varphi_{\bar a}(\bar b)$ iff $\bar b$ lies in the orbit of
$\bar a$.

For the complexity: the diagram, $\mathrm{root}$, $\mathrm{lev}_2$,
$\dg_{\geq k}$, $\mathrm{samepar}$ are in $\Sigma_1\cup\Pi_1$; the
$\mathrm{d}\text{-}\Sigma_1$ clauses $\mathrm{lev}_1,\dg_{=k}$ and  $\mathrm{pardeg}_{=k}$ are in $\Sigma_2$; $\dg_{=\omega}$ is
$\Pin2$; and the unique $\Sin3$ clause is $\mathrm{pardeg}_{=\omega}$, which
occurs only positively. All of these are contained in $\Sin3$, no conjunct is $\Pin3$ or higher, and
$\Sin3$ is closed under finite conjunction; hence $\varphi_{\bar a}\in\Sin3$.

Hence $\{\varphi_{\bar a}\colon\bar a \text{ a tuple in }\A\}$ is a parameterless $\Sin{3}$ Scott family for
$\A$. By Theorem~\ref{thm:montalban} ((2)$\Rightarrow$(1)), $\A$ has a
$\Pin{4}$ Scott sentence, and $\SR(\A)\leq 3$ directly from the definition of
$\SR$.
\end{proof}

\subsection{$O_\omega$ is not $\Sin{2}$-definable}

\begin{theorem}\label{thm:notSigma2}
There is no $\Sin{2}$ formula $\varphi(x)$ such that for all $u\in\A$: $\A\models\varphi(u)$ iff $u\in O_\omega$. Together with \S\ref{ssec:blocks}, the orbit $O_\omega$ is $\Pin{2}$- but not $\Sin{2}$-definable in $\A$.
\end{theorem}

\begin{proof}
Suppose $\varphi$ is such a formula; write it in normal form $\varphi(x)=\bigvee_i\exists\bar y_i\,\psi_i(x,\bar y_i)$ with each $\psi_i\in\Pin{1}$, i.e.\ a countable conjunction of finitary universal formulas with quantifier-free matrix.

Fix $u\in O_\omega$. Since $\A\models\varphi(u)$, some disjunct $\exists\bar y\,\psi$
with $\psi:=\psi_i\in\Pin{1}$ is witnessed by a finite tuple $\bar a$, i.e.\
$\A\models\psi(u,\bar a)$. We may assume $u\notin\bar a$: if
$J:=\{\,j : a_j=u\,\}$,
replace each variable $y_j$ with $j\in J$ throughout $\psi$ by the displayed
variable $x$. The resulting formula $\psi^*(x,\bar y^*)$, where
$\bar y^*:=(y_k)_{k\notin J}$, is again $\Pin{1}$ (a variable-for-variable
substitution leaves the universal prefix and quantifier-free matrix intact);
and since $x$ and every replaced $y_j$ are interpreted by the same element $u$,
we have $\A\models\psi^*(u,\bar a^*)$ with $\bar a^*:=(a_k)_{k\notin J}$. By
construction $a_k\neq u$ for every $k\notin J$, so $u\notin\bar a^*$. Replacing
$(\psi,\bar a)$ by $(\psi^*,\bar a^*)$ gives a disjunct $\psi\in\Pin{1}$ and a
witness $\bar a$ with $u\notin\bar a$.

The strategy now is to build, using the Extension Lemma~\ref{lem:ext}, a leveled embedding $h$ that sends $u$ to a fresh level-one node $v$ of finite degree. The Claim below implies that $\A\models\psi(v,h(\bar a))$ (the $\Pin1$ disjunct $\psi$ is preserved), and hence $\A\models\varphi(v)$. But, as $\dg(v)$ is
finite ($v\notin O_\omega$), we get the desired contradiction.

Let $c:=|\bar a\cap\ch(u)|$ and let $P$ be the (finite) set of parents of level-two elements of $\bar a\setminus\ch(u)$. Choose a level-one node $v$ of finite degree $N\geq c$ that is fresh for the configuration, i.e., $v\notin\bar a\cup\{u\}\cup P$. Such $v$ exists since each $O_N$ is infinite. Define $h$ on $\{\rho,u\}\cup\bar a\cup P$ by: $h(u):=v$; $h$ maps the $c$ elements of $\bar a\cap\ch(u)$ to $c$ distinct children of $v$; $h$ is the identity on $\rho$, on $\bar a\setminus\ch(u)$, and on $P$. A case check on pairs shows $h$ is a leveled embedding: the only comparabilities involving relocated elements are $\rho\prec u\prec a$ for $a\in\bar a\cap\ch(u)$, matched by $\rho\prec v\prec h(a)$; all other pairs involving $u$ or its named children are incomparable on both sides, by freshness of $v$ and of its chosen children.

\begin{claim}\label{claim:pi1transfer}
Every $\Pin{1}$ formula true of $(u,\bar a)$ in $\A$ is true of $(v,h(\bar a))$ in $\A$.
\end{claim}

The point of the Claim is to transport the chosen disjunct $\psi$ from the configuration at $u$ to its image at $v$. This is \emph{not} automatic from $h$ being a leveled embedding. The map $h$ has finite domain $\{\rho,u\}\cup\bar a\cup P$, so by itself it preserves only quantifier-free facts about tuples in that domain; first-order---let alone infinitary---formulas are not preserved, since their quantifiers range over all of $\A$, not over $\dom(h)$.

\begin{proof}[Proof of Claim]
It suffices to treat one conjunct $\forall\bar z\,\theta$ ($\theta$ quantifier-free) at a time. Suppose $\A\models\forall\bar z\,\theta(u,\bar a,\bar z)$ but $\A\models\neg\theta(v,h(\bar a),\bar e)$ for some tuple $\bar e$. Apply the Extension Lemma~\ref{lem:ext} with $A=B=\A$, with
\[
X:=\{\rho,v\}\cup h(\bar a)\cup P, \qquad E:=X\cup\bar e, \qquad p:=h^{-1}\ \text{(so } p(v)=u\text{)} .
\]
Notice that $p$ has domain exactly $X$. Hypothesis (i) holds: every level-two element of $X$ is either a named child of $v$ (the parent is in $X$) or an element of $h(\bar a)\setminus\ch(v) = \bar a\setminus\ch(u)$ (parent is in $P\subseteq X$ if the parent lies in $E$ at all). Hypothesis (ii) holds: $p(v)=u$ has degree $\omega$; every other level-one element of $X$ is mapped by the identity, for which (ii) is trivial since $E\cap\ch(w)\subseteq\ch(w)$. We obtain a leveled embedding $q\colon E\to\A$ extending $p$, so $q(v)=u$ and $q(h(a))=a$ for all $a$. Since $\theta$ is quantifier-free in the relational language $\{\prec\}$ and $q$ preserves and reflects $\prec$ and equality, $\A\models\neg\theta(u,\bar a,q(\bar e))$ --- contradicting $\A\models\forall\bar z\,\theta(u,\bar a,\bar z)$. (The decisive asymmetry is hidden in hypothesis (ii): pulling the configuration back \emph{onto} $u$ never exhausts the child supply because $\dg(u)=\omega$, whereas the forward direction onto a finite-degree node would fail for large $\bar e$. This is the whole reason the lemma, and the theorem, are true.)
\end{proof}

By the Claim, $\A\models\psi(v,h(\bar a))$, hence $\A\models\varphi(v)$, hence $v\in O_\omega$; but $\dg(v)=N$ is finite. Contradiction.
\end{proof}

\begin{remark}\label{rem:localize}
Theorem~\ref{thm:notSigma2} localizes the orbit produced abstractly by Theorem~\ref{thm:SR} (case $m=1$): the coding orbit $O_\omega$ of Remark~\ref{rem:coding-orbit} is itself the non-$\Sin{2}$-definable one. The proof uses only that every finite degree occurs infinitely often among level-one nodes and that members of $O_\omega$ have infinitely many children; it therefore applies verbatim in $[2:j]$ for every $j\geq 1$.
\end{remark}

\subsection{No $\Pin{3}$ and no $\dSin{3}$ Scott sentence}

For $j\in\omega$ let $B_j:=[2:j]$, realized inside $\A$ as follows: enumerate $O_\omega=\{u_0,u_1,\dots\}$ and let $B_j$ be the induced subtree of $\A$ on $\A\setminus\bigcup_{i\geq j}\big(\{u_i\}\cup\ch(u_i)\big)$. Then $B_j\cong[2:j]$ and $B_j\not\cong\A$.

\begin{lemma}\label{lem:transfers}
\begin{enumerate}
\item[(i)] Every $\Pin{3}$ sentence true in $\A$ is true in $B_j$, for every $j\in\omega$.
\item[(ii)] For every $\Sin{3}$ sentence $\sigma$ true in $\A$ there is $j_\sigma\in\omega$ such that $B_j\models\sigma$ for all $j\geq j_\sigma$.
\end{enumerate}
\end{lemma}

\begin{proof}
Both parts rest on instances of the Extension Lemma~\ref{lem:ext}; note that both $\A$ and every $B_j$ have, for every $d$, infinitely many level-one nodes of degree $\geq d$, so the lemma could possibly be applied in either direction.

(ii) Write $\sigma=\bigvee_i\exists\bar y_i\,\theta_i$ with $\theta_i\in\Pin{2}$. 
Fix a disjunct $\theta$ and witnesses $\bar a\subseteq\A$ with $\A\models\theta(\bar a)$.
As $\bar a$ is finite, $S:=\{i:(\{u_i\}\cup\ch(u_i))\cap\bar a\neq\emptyset\}$ is finite;
put $j_\sigma:=1+\max S$, with $j_\sigma:=0$ if $S=\emptyset$. For $j\geq j_\sigma$ every
$i\in S$ has $i<j$, so $u_i$ and its children survive in $B_j$; hence $\bar a\subseteq B_j$. Fix such a $j$; we check that $B_j\models\theta(\bar a)$, which will imply that $B_j\models\sigma$.

Write $\theta=\bigwedge\forall\bar z\,\chi$ with $\chi\in\Sin{1}$, i.e.\ $\chi=\bigvee\exists\bar w\,(\text{quantifier-free})$. Let $\bar e\subseteq B_j\subseteq\A$ be arbitrary. Since $\A\models\chi(\bar a,\bar e)$, there are quantifier-free witnesses $\bar w\subseteq\A$. Apply the Extension Lemma with $A:=\A$, $B:=B_j$, $E:=\{\rho\}\cup\bar a\cup\bar e\cup\bar w\cup P$, $X:=\{\rho\}\cup\bar a\cup\bar e\cup P$, and $p:=$ the identity, where $P\subseteq B_j$ is the set of parents of the level-two elements of $\bar a\cup\bar e$ (note $P\subseteq B_j$ since $B_j$ removes whole level-one subtrees). Hypothesis (i) of the lemma holds by the inclusion of $P$; hypothesis (ii) is trivial for the identity, since the children of a surviving level-one node are the same in $B_j$ as in $\A$. The resulting $q\colon E\to B_j$ is the identity on $\bar a\bar e$ and plants order-isomorphic witnesses $q(\bar w)$ in $B_j$, so $B_j\models\chi(\bar a,\bar e)$. As $\bar e$ was arbitrary, $B_j\models\theta(\bar a)$.

(i) Write $\pi=\bigwedge\forall\bar y\,\chi$ with $\chi\in\Sin{2}$, true in $\A$. Fix $\bar c\subseteq B_j\subseteq\A$; we must show $B_j\models\chi(\bar c)$. Since $\A\models\chi(\bar c)$, some disjunct $\exists\bar w\,\psi$ with $\psi\in\Pin{1}$ has witnesses $\bar w\subseteq\A$ --- possibly inside removed subtrees. Define a leveled embedding $h$ on $\{\rho\}\cup\bar c\cup\bar w\cup P$ ($P$ the parents, within that set's level-two part, lying in $\A$) into $B_j$ as follows: $h$ relocates each removed root $u_i$ ($i\geq j$) occurring in the domain to a distinct fresh level-one node of $B_j$ of finite degree $\geq|\bar c\cup \bar w|$, relocates the named children of each such $u_i$ to distinct fresh children of its image, and is the identity elsewhere; this is an instance of the Extension Lemma (with $A:=\A$, $B:=B_j$, $X:=\{\rho\}\cup\bar c\cup$ the surviving part, extended over the removed part), or can be checked directly as in Theorem~\ref{thm:notSigma2}.

\begin{claim}
$B_j\models\psi(\bar c,h(\bar w))$.
\end{claim}

\begin{proof}[Proof of Claim]
As in Claim~\ref{claim:pi1transfer}, fix a conjunct $\forall\bar z\,\theta_0$ of $\psi$
($\theta$ quantifier-free) and suppose, for contradiction, that
$B_j\models\neg\theta(\bar c,h(\bar w),\bar e)$ for some $\bar e\subseteq B_j$.
Let $P'$ be the set of $B_j$-parents of the level-two elements of $\bar c\cup h(\bar w)$,
and apply the Extension Lemma~\ref{lem:ext} with
\[
  A:=B_j,\qquad B:=\A,\qquad X:=\{\rho\}\cup\bar c\cup h(\bar w)\cup P',\qquad E:=X\cup\bar e,
\]
and $p\colon X\to\A$ given by $p:=h^{-1}$ on $\operatorname{ran}(h)\cap X$ and $p:=\mathrm{id}$
on $X\setminus\operatorname{ran}(h)$. These agree where both apply ($h$ is the identity off
the relocated nodes) and do not collide ($u_i\notin B_j\supseteq\bar c\cup P'$), so $p$ is a
well-defined leveled embedding; in particular each relocated image $v_i:=h(u_i)\in P'$ is sent
to $u_i$ \emph{(not to itself)}, which is what preserves $\prec$ on the pairs $v_i\prec h(t)$
for $t\in\ch(u_i)\cap\bar w$.

Hypothesis~(i) holds by the choice of $P'$: every level-two element of $X$ lies in
$\bar c\cup h(\bar w)$, and its $B_j$-parent is either a relocated image in $h(\bar w)$ or an
element of $P'$, hence in $X$. Hypothesis~(ii) holds for every level-one $w\in X$: if $w=v_i$
is a relocated image then $p(w)=u_i$ has degree $\omega\ge|E\cap\ch_{B_j}(v_i)|$; otherwise $w$
is a surviving node, $p(w)=w$, and $\dg_\A(w)=\dg_{B_j}(w)=|\ch_{B_j}(w)|\ge|E\cap\ch_{B_j}(w)|$,
since $B_j$ retains all children of surviving level-one nodes.

The lemma yields a leveled embedding $q\colon E\to\A$ extending $p$, so $q(\bar c)=\bar c$ and
$q(h(\bar w))=\bar w$. As $\theta$ is quantifier-free and $q$ preserves and reflects $\prec$
and equality, $\A\models\neg\theta(\bar c,\bar w,q(\bar e))$ --- contradicting the conjunct
$\forall\bar z\,\theta(\bar c,\bar w,\bar z)$ of $\A\models\psi(\bar c,\bar w)$.
\end{proof}
\end{proof}

\begin{theorem}\label{thm:rank2}
$\SSC(\A_2)=\Pin{4}$ and $\SR(\A_2)=\pSR(\A_2)=3$. In detail: $\A_2$ has a $\Pin{4}$ Scott sentence but no $\dSin{3}$ Scott sentence---hence none in $\Sin{3}$, $\Pin{3}$, or (via Theorem~\ref{thm:miller}) $\Sin{4}$.
\end{theorem}

\begin{proof}
In the ordering $\Sin\beta,\Pin\beta<\dSin\beta<\Sin{\beta+1},\Pin{\beta+1}$ the
classes strictly below $\Pin4$ are exactly the nine classes
$\Sin\beta,\Pin\beta,\dSin\beta$ ($\beta\in\{1,2,3\}$), and the unique class
\emph{incomparable} to $\Pin4$ is $\Sin4$.

\emph{$\A$ has a $\Pin4$ Scott sentence:} Proposition~\ref{prop:scottfamily}.

\emph{No $\Pin3$ Scott sentence:} such a sentence is true in $\A$ and, being $\Pin3$,
true in $B_0$ by Lemma~\ref{lem:transfers}(i); but $B_0\not\cong\A$.

\emph{No Scott sentence in any class included in $\Pin3$:} each of $\Sin\beta,\Pin\beta,
\dSin\beta$ ($\beta\in\{1,2\}$) and $\Pin3$ itself is included in $\Pin3$ --- for
$\dSin2=\Sin2\wedge\Pin2$ use that $\Sin2,\Pin2\subseteq\Pin3$ and $\Pin3$ is closed
under finite conjunction --- so a Scott sentence in such a class is a $\Pin3$ Scott
sentence, excluded above.

\emph{No $\Sin3$ Scott sentence:} by Lemma~\ref{lem:transfers}(ii) it would hold in
$B_j\not\cong\A$ for all $j\geq j_\sigma$.

\emph{No $\dSin3$ Scott sentence:} such a sentence is $\sigma\wedge\pi$ with
$\sigma\in\Sin3$, $\pi\in\Pin3$; if true in $\A$ then $\pi$ holds in every $B_j$ by
Lemma~\ref{lem:transfers}(i) and $\sigma$ holds in $B_j$ for $j\geq j_\sigma$ by
Lemma~\ref{lem:transfers}(ii), so $\sigma\wedge\pi$ holds in $B_j\not\cong\A$ for every
$j\geq j_\sigma$.

\emph{No $\Sin4$ Scott sentence:} a $\Sin4$ Scott sentence together with the $\Pin4$
one above would, by Theorem~\ref{thm:miller} (successor case $\alpha=4$, giving
$\dSin3$), contradict the previous paragraph.

\emph{Conclusion.} Let $C$ be any complexity realized by a Scott sentence of $\A$. If
$C=\Sin4$, or if $C$ is strictly below $\Pin4$, it has just been excluded; every
remaining class satisfies $C\succeq\Pin4$. As $\Pin4$ is itself realized, it is the
least realized complexity: $\SSC(\A_2)=\Pin4$.

By Theorem~\ref{thm:montalban}, $\SR$ is the least $\alpha$ with a $\Pin{\alpha+1}$
Scott sentence: the $\Pin4$ sentence gives $\SR(\A_2)\leq 3$ and the absence of a
$\Pin3$ one gives $\SR(\A_2)\geq 3$, so $\SR(\A_2)=3$. Finally
\cite[Table~1]{GR23} gives $\SSC=\Pin{\alpha+1}\Rightarrow\pSR=\SR=\alpha$ with no
parameters; applied at the successor $\alpha=3$, $\pSR(\A_2)=3$.
\end{proof}

\begin{remark}
Theorem~\ref{thm:rank2} and Lemma~\ref{lem:transfers} also re-prove the lower bound of Theorem~\ref{thm:SR} at $m=1$ with no computability theory at all; the two routes --- relativized categoricity plus the transfer principle, and direct back-and-forth --- agree exactly, as the Montalb\'an correspondence predicts. We expect the same dual structure at every finite rank (Section~\ref{sec:questions}).
\end{remark}

\subsection{Effectivity: computable Scott sentences and index sets}

\begin{proposition}\label{prop:computable-ss}
There is a computable copy $\hat\A$ of $\A_2$ admitting a c.e.\ Scott family of $\Sc{3}$ formulas. Consequently $\A_2$ has a computable $\Pc{4}$ Scott sentence, and the computable Scott sentence complexity of $\A_2$ is exactly $\Pc{4}$.
\end{proposition}

\begin{proof}
Build $\hat\A$ with computable structural data: universe $\omega$, root $0$;
level-one nodes $a_{k,j}$ ($k,j\in\omega$) and $b_j$ ($j\in\omega$) under a fixed
computable coding; children $a_{k,j,1},\dots,a_{k,j,k}$ of $a_{k,j}$ and
$b_{j,i}$ ($i\in\omega$) of $b_j$. Then $a_{k,j}$ has degree $k$ and $b_j$ has
degree $\omega$, each degree occurring infinitely often, so $\hat\A\cong[2:\omega]=\A_2$.

\emph{A uniformly computable Scott family.} From the code of a tuple $\bar a$ one
computes its level, degree, parenthood, and shared-parent data --- the datum of
Lemma~\ref{lem:datum} --- and hence (Proposition~\ref{prop:scottfamily}) an index
for the formula $\varphi_{\bar a}$. This map is total computable, with
$\dg_{=\omega}=\bigwedge_k\dg_{\ge k}$ and $\mathrm{pardeg}_{=\omega}=\exists w(\dots\wedge\dg_{=\omega}(w))$
furnished as computable infinitary $\Pc2$ and $\Sc3$ formulas. As each
$\varphi_{\bar a}$ is $\Sc3$ (Proposition~\ref{prop:scottfamily}) and occurs with no
parameters, $\{\varphi_{\bar a}\}$ is a uniformly computable --- in particular c.e.\ ---
parameterless $\Sc3$ Scott family for $\hat\A$.

\emph{A computable $\Pc4$ Scott sentence.} By \cite[Prop.~2.9]{AKM20}, such a family
yields a computable $\Pc4$ Scott sentence: the conjunction of $\varphi_\emptyset$ with
the sentences $\forall\bar x\,[\varphi_{\bar a}(\bar x)\to(\forall y\bigvee_b\varphi_{\bar a b}(\bar x,y)\wedge\bigwedge_b\exists y\,\varphi_{\bar a b}(\bar x,y))]$,
which is $\Pc4$ since each $\varphi_{\bar a b}\in\Sc3$.

\emph{Optimality, and that it is a property of $\A_2$.} Whether the isomorphism type
$\A_2$ admits a computable Scott sentence of a given complexity does not depend on
the chosen computable copy: a Scott sentence characterizes the type among countable
structures, and computability of an infinitary sentence is a syntactic property of
the sentence. The copy $\hat\A$ witnesses a computable $\Pc4$ Scott sentence. For the
lower bound, a computable Scott sentence is in particular a lightface, hence boldface,
Scott sentence of the same class; every complexity properly below $\Pc4$ is contained
in $\dSin3$, and the only sibling $\Sc4$ embeds in $\Sin4$, so any computable Scott
sentence beating $\Pc4$ would furnish a boldface $\dSin3$ or $\Sin4$ Scott sentence,
both excluded by Theorem~\ref{thm:rank2}. Hence $\A_2$ has computable Scott sentence
complexity exactly $\Pc4$, matching its boldface value $\SSC(\A_2)=\Pin4$.
\end{proof}

\begin{remark}
The boldface and computable optimal complexities thus coincide for $\A_2$. This is not automatic: Alvir, Knight, and McCoy \cite{AKM20} constructed a computable structure with a $\Pin{2}$ Scott sentence but no computable $\Pc{2}$ Scott sentence, and Alvir, Csima, and Harrison-Trainor \cite{ACH25} showed the generic gap (a $\Pin{\alpha}$ Scott sentence yields a computable $\Pc{2\alpha}$ one) is sharp in general. For the trees considered here the exact match is delivered by the c.e.\ Scott family, i.e.\ by the homogeneity of the types $[n:N]$, just as in the effectivity remarks of \cite{GR23} for linear orders.
\end{remark}

\begin{corollary}\label{cor:indexsets}
The isomorphism class $\{\B\in\Mod(\{\prec\}):\B\cong\A_2\}$ is (boldface) $\bPi{4}$ and not (boldface) $\bSig{4}$. Lightface, the index set $\{e: \mathcal{M}_e\cong\A_2\}$ (over any computable enumeration of computable $\{\prec\}$-structures) is $\Pi^0_4$-complete.
\end{corollary}

\begin{proof}
Boldface: $\bPi{4}$ is Theorem~\ref{thm:montalban}(3) with $\alpha=3$, since $\A_2$ has a
$\Pin{4}$ Scott sentence (Proposition~\ref{prop:scottfamily}). If the class were $\bSig{4}$,
then by the level-by-level L\'opez-Escobar/Vaught theorem \cite{Vau74} (invariant
$\bSig{4}$ classes are exactly the $\Sin{4}$-axiomatizable ones) $\A_2$ would have a
$\Sin{4}$ Scott sentence, contradicting Theorem~\ref{thm:rank2}.

Lightface: membership is $\Pi^0_4$. Verifying that $\mathcal M_e$ is a rank-$2$ tree
(strict partial order, finite linearly ordered predecessor sets, a root, no chain of
length $>2$) is $\Pi^0_3$, and against a fixed computable copy of $\A_2$ ``$\mathcal M_e
\cong \A_2$'' is $\Pi^0_4$ by \cite[Lem.~2.1]{Mah19} (isomorphism of rank-$2$ trees is
$\Pi^0_4$, uniformly); the conjunction is $\Pi^0_4$, the binding cost being the condition
that level-one nodes of infinite degree occur infinitely often. Hardness is
\cite[Lem.~3.2]{Mah19} at $n=1$: it sends any $\Sigma^0_4$ set $D$ to a uniformly
computable sequence $(\T^x)$ of rank-$2$ trees with $\T^x\cong[2:\omega]$ iff $x\notin D$;
taking $D$ a $\Sigma^0_4$-complete set, $x\mapsto\T^x$ reduces the $\Pi^0_4$-complete set
$\overline D$ to the index set. Hence the index set is $\Pi^0_4$-complete.
\end{proof}

\section{The general case: $\SSC(\A_{m+1})=\Pin{2m+2}$}\label{sec:general}

The rank-$2$ analysis of Section~\ref{sec:ranktwo} pins $\SSC(\A_2)=\Pin{4}$ by a back-and-forth argument tailored to rank $2$. We now prove the general statement by a different and shorter route, due to the second author, which bypasses the substitution lemma of Section~\ref{sec:questions}. The idea is that the relativized coding of Section~\ref{sec:relhard} is robust under naming parameters: a finite tuple meets only finitely many top-level subtrees of $\A_{m+1}$, and after reserving those there remain infinitely many spare top-level subtrees of every relevant type on which to run the coding. This yields $\SR(\A_{m+1},\bar a)=2m+1$ for \emph{every} finite tuple $\bar a$, hence $\pSR(\A_{m+1})=2m+1$, which the dictionary of \cite[Table~1]{GR23} converts into $\SSC(\A_{m+1})=\Pin{2m+2}$.

Throughout, $(\A_{m+1},\bar a)$ denotes the expansion of $\A_{m+1}$ by constants naming the entries of a finite tuple $\bar a$ (repeats allowed). Recall from Remark~\ref{rem:coding-orbit} that the coding rides on the orbit
\[
O^{m}_\omega=\{u:\Lev(u)=1 \text{ and } T[u]\cong[m:\omega]\},
\]
which is $\Pi^0_{2m}(Z)$-complete in the coding copy of Theorem~\ref{thm:relhard} and $Z$-computably marked --- by parity of the level-one root --- in the transparent copy.

\begin{definition}\label{def:support}
For a finite tuple $\bar a$ from $\A_{m+1}$, the \emph{top-level support} of $\bar a$ is the finite set
\[
\mathrm{supp}(\bar a)=\{u:\Lev(u)=1 \text{ and some coordinate of }\bar a\text{ lies in }\A_{m+1}[u]\}.
\]
A coordinate equal to the root contributes nothing: the root is fixed by every automorphism. For $u\in\mathrm{supp}(\bar a)$ let $\bar a_u$ be the subtuple of coordinates lying in $\A_{m+1}[u]$; the \emph{pointed top-level component} at $u$ is $(\A_{m+1}[u],\bar a_u)$.
\end{definition}

\begin{lemma}[Reservation lemma]\label{lem:reserve}
Fix finite $m\geq 1$ and a finite tuple $\bar a\in\A_{m+1}^{<\omega}$. For every oracle $Z\subseteq\omega$ there are $Z$-computable copies $\B_Z(\bar a),\Ccal_Z(\bar a)\cong(\A_{m+1},\bar a)$ of the following form.
\begin{enumerate}
\item Finitely many top-level subtrees are \emph{reserved} in each copy; the reserved subtrees realize exactly the pointed top-level components $(\A_{m+1}[u],\bar a_u)$ for $u\in\mathrm{supp}(\bar a)$, and the constants naming $\bar a$ are interpreted inside them (a coordinate equal to the root is interpreted as the root).
\item Outside the reserved subtrees, the \emph{ordinary} part of each copy is the relativized coding construction of Theorem~\ref{thm:relhard}: in $\B_Z(\bar a)$ the ordinary designated level-one roots code a set Turing-equivalent to $Z^{(2m)}$, and in $\Ccal_Z(\bar a)$ membership of an ordinary level-one root in $O^{m}_\omega$ is $Z$-computably marked.
\end{enumerate}
\end{lemma}

\begin{proof}
For each $u\in\mathrm{supp}(\bar a)$ reserve one fresh top-level subtree in each copy and make it a computable copy of the pointed component $(\A_{m+1}[u],\bar a_u)$. This is possible because every type $[r:N]$ with $r\leq m$ and $N\in\omega\cup\{\omega\}$ has a computable presentation (Definition~\ref{def:types}), and a finite pointed configuration inside such a type is realized in a computable copy by naming computable representatives of the finitely many coordinates involved. Outside the reserved subtrees, attach the ordinary top-level subtrees of the copies $\T_Z,\Scal_Z$ of Theorem~\ref{thm:relhard}.

Since $\A_{m+1}=[m+1:\omega]$ has infinitely many top-level subtrees of type $[m:\omega]$ and infinitely many of type $[m:k]$ for every finite $k$, deleting the finitely many reserved ones leaves the same isomorphism type; so the underlying tree of each copy is again $\A_{m+1}$, and interpreting the constants inside the reserved pieces makes each copy a copy of $(\A_{m+1},\bar a)$. The reserved part is computable outright and the ordinary part is $Z$-computable by Theorem~\ref{thm:relhard}, so each copy is $Z$-computable. Clauses (1) and (2) hold by construction, the marking in $\Ccal_Z(\bar a)$ being the parity convention of Theorem~\ref{thm:relhard} applied to the ordinary roots.
\end{proof}

\begin{theorem}[Parameterized relativized coding]\label{thm:paramcoding}
Let $m\geq 1$ be finite and $\bar a\in\A_{m+1}^{<\omega}$. For every oracle $Z\subseteq\omega$ there are $Z$-computable copies $\B_Z(\bar a),\Ccal_Z(\bar a)\cong(\A_{m+1},\bar a)$ such that every isomorphism $f\colon\B_Z(\bar a)\to\Ccal_Z(\bar a)$ satisfies $Z^{(2m)}\leT f$.
\end{theorem}

\begin{proof}
Take the copies of Lemma~\ref{lem:reserve}; let $R_\B,R_\Ccal$ be their finite sets of reserved top-level roots. Let $r_x$ be an ordinary designated coding root of $\B_Z(\bar a)$, so $r_x\notin R_\B$.

\emph{An isomorphism $f$ cannot send $r_x$ into the reserved region of $\Ccal_Z(\bar a)$.} Suppose $f(r_x)=s$ with $s\in R_\Ccal$. The reserved subtree at $s$ contains a named coordinate $a^\Ccal_i$; as $f$ respects the constants, $f(a^\B_i)=a^\Ccal_i$, where $a^\B_i$ is the corresponding coordinate of $\B_Z(\bar a)$. Now $a^\Ccal_i\preceq s=f(r_x)$ or $a^\Ccal_i=s=f(r_x)$, and an isomorphism reflects $\preceq$, so $a^\B_i\preceq r_x$ or $a^\B_i=r_x$. But $a^\B_i$ lies in a reserved subtree while $r_x$ is an ordinary root, and distinct top-level subtrees share only the global root; the order relation $a^\B_i\preceq r_x$ is therefore impossible, while $a^\B_i=r_x$ contradicts $a^\B_i\in R_\B$, $r_x\notin R_\B$. Hence $f(r_x)$ is an \emph{ordinary} level-one root of $\Ccal_Z(\bar a)$.

\emph{Decoding.} Since $f$ is an isomorphism,
\[
\B_Z(\bar a)[r_x]\cong[m:\omega]\iff \Ccal_Z(\bar a)[f(r_x)]\cong[m:\omega],
\]
that is, $r_x\in O^{m}_\omega\iff f(r_x)\in O^{m}_\omega$. By the previous paragraph $f(r_x)$ is ordinary, and there membership in $O^{m}_\omega$ is $Z$-computably marked (Lemma~\ref{lem:reserve}(2)); so $f$ computes the coding set carried by the roots $r_x$. By Theorem~\ref{thm:relhard} that set is $\mathrm{Inf}^Z$ for $m=1$ and the complement of $Z^{(2m)}$ for $m\geq 2$, each $\Pi^0_{2m}(Z)$-complete under $\leq_1$; hence $Z^{(2m)}\leT f$.
\end{proof}

\begin{proposition}[Transfer with finite parameters]\label{prop:transfer-param}
Let $\mathcal M$ be a countable structure in a relational language, expanded by finitely many constants, and let $n\geq 1$ be finite. Suppose that for every $Z\subseteq\omega$ there are $Z$-computable copies $\mathcal M^0_Z,\mathcal M^1_Z\cong\mathcal M$ between which no isomorphism is $Z^{(n-1)}$-computable. Then $\SR(\mathcal M)\geq n+1$.
\end{proposition}

\begin{proof}
This is Proposition~\ref{prop:transfer} for the finite expanded language. Adjoining finitely many constants is, for the back-and-forth of Claim~\ref{claim:bf}, the same as adjoining finitely many unary predicates naming singletons: atomic diagrams of the copies remain decidable, and the construction of Claim~\ref{claim:bf} only ever \emph{finds} $\Pi_{n-1}$ witnesses, never deciding $\Sin{n}$ truth, so it goes through verbatim with the constants present. A putative parameterless $\Sin{n}$ Scott family for the expansion, coded by a real $X_0$, would by that argument yield, between any two copies, an isomorphism computable in $X_0^{(n-1)}$; taking $Z=X_0$ contradicts the hypothesis. So the expansion has no parameterless $\Sin{n}$ Scott family, and by Theorem~\ref{thm:montalban}, $\SR(\mathcal M)\geq n+1$.
\end{proof}

\begin{theorem}\label{thm:paramSR}
For every finite $m\geq 1$ and every finite tuple $\bar a\in\A_{m+1}^{<\omega}$,
\[
\SR(\A_{m+1},\bar a)=2m+1.
\]
\end{theorem}

\begin{proof}
\emph{Lower bound.} By Theorem~\ref{thm:paramcoding}, for every $Z$ there are $Z$-computable copies of $(\A_{m+1},\bar a)$ between which every isomorphism computes $Z^{(2m)}$; since $Z^{(2m)}\not\leT Z^{(2m-1)}$, none is $Z^{(2m-1)}$-computable. Proposition~\ref{prop:transfer-param} with $n=2m$ gives $\SR(\A_{m+1},\bar a)\geq 2m+1$.

\emph{Upper bound.} By Theorem~\ref{thm:SR} every automorphism orbit of a finite tuple of $\A_{m+1}$ is defined by a parameter-free $\Sin{2m+1}$ formula. Let $\bar b$ be a finite tuple of the expansion $(\A_{m+1},\bar a)$ and let $\varphi(\bar y,\bar x)\in\Sin{2m+1}$ define the orbit of the concatenation $(\bar a,\bar b)$ in $\A_{m+1}$. Substituting the constants for $\bar y$ yields $\varphi(\bar a,\bar x)$, still $\Sin{2m+1}$, which defines exactly the orbit of $\bar b$ in $(\A_{m+1},\bar a)$, because the automorphisms of the expansion are precisely the automorphisms of $\A_{m+1}$ fixing $\bar a$. Hence every orbit of the expansion is $\Sin{2m+1}$-definable and $\SR(\A_{m+1},\bar a)\leq 2m+1$.
\end{proof}

\begin{corollary}\label{cor:pSR}
For every finite $m\geq 1$, $\pSR(\A_{m+1})=2m+1$.
\end{corollary}

\begin{proof}
$\pSR(\A_{m+1})$ is the least value of $\SR(\A_{m+1},\bar a)$ over finite tuples $\bar a$; by Theorem~\ref{thm:paramSR} every such value equals $2m+1$.
\end{proof}

\begin{theorem}\label{thm:main-general}
For every finite $m\geq 1$, $\SSC(\A_{m+1})=\Pin{2m+2}$.
\end{theorem}

\begin{proof}
By Corollary~\ref{cor:dichotomy}, $\SSC(\A_{m+1})$ is one of $\Sin{2m+1}$, $\dSin{2m+1}$, $\Pin{2m+2}$. We read off the parameterized rank forced by each candidate from \cite[Table~1]{GR23}, using $\SR(\A_{m+1})=2m+1$ (Theorem~\ref{thm:SR}). A $\Pin{\alpha+1}$ Scott sentence gives $\pSR=\SR=\alpha$, so $\Pin{2m+2}$ forces $\pSR=2m+1$. A $\dSin{\alpha+1}$ Scott sentence gives $\pSR=\alpha$, so $\dSin{2m+1}$ forces $\pSR=2m$. A $\Sin{\alpha+2}$ Scott sentence gives $\pSR=\alpha$ (with $\SR=\alpha+2$), so $\Sin{2m+1}=\Sin{(2m-1)+2}$ forces $\pSR=2m-1$. The three candidates thus force the three \emph{distinct} values $2m-1,\,2m,\,2m+1$ of the parameterized rank; consequently $\SSC(\A_{m+1})=\Pin{2m+2}$ if and only if $\pSR(\A_{m+1})=2m+1$. Corollary~\ref{cor:pSR} supplies $\pSR(\A_{m+1})=2m+1$, so $\SSC(\A_{m+1})=\Pin{2m+2}$.
\end{proof}

\begin{remark}\label{rem:general-route}
The dictionary step above corrects a slip in the candidate triple's bookkeeping: the $\Sin{2m+1}$ case forces $\pSR=2m-1$, not $2m$, the $\Sigma$-index exceeding the parameterized rank by two (\cite[Table~1]{GR23}). The value is immaterial to the elimination, which needs only that the three candidates give three distinct parameterized ranks; at $m=1$ these are $1,2,3$, matching the proven rank-$2$ values $\SSC([2:0])=\Pin{3}$, $\SSC([2:j])=\dSin{3}$ ($1\leq j<\omega$), and $\SSC(\A_2)=\Pin{4}$. Unlike the rank-$2$ argument of Section~\ref{sec:ranktwo}, the proof here uses computability essentially --- through the relativized coding of Theorem~\ref{thm:relhard} and the transfer principle --- and it does not localize the non-$\Sin{2m}$-definable orbit. The substitution-lemma program of Section~\ref{sec:questions} would give a parallel, computability-free proof that also pins the critical orbit at $O^{m}_\omega$.
\end{remark}

\section{The substitution lemma and open questions}\label{sec:questions}

Theorem~\ref{thm:main-general} settles the trichotomy of Corollary~\ref{cor:dichotomy} in favour of the $\Pi$ case, by way of the parameterized rank: among the three candidates only $\Pin{2m+2}$ forces $\pSR(\A_{m+1})=2m+1$ (the others force $2m-1$ and $2m$; see the proof of Theorem~\ref{thm:main-general}). That proof uses computability essentially, through the relativized coding of Section~\ref{sec:relhard}, and does not localize the critical orbit. We record here a second, purely combinatorial route, generalizing the rank-$2$ analysis directly: it would re-prove Theorem~\ref{thm:main-general} without computability and pin the non-$\Sin{2m}$-definable orbit at $O^{m}_\omega$. It remains open for $m\geq 2$.

A $\Pi$/not-$\Pi$ \emph{dichotomy} would be settled by a single back-and-forth substitution lemma, as at rank~$2$. The \emph{trichotomy} of Corollary~\ref{cor:dichotomy} requires excluding $\Sin{2m+1}$ and $\dSin{2m+1}$ separately, and these are governed by transfers that bottom out at \emph{different} quantifier levels; we therefore split the rank-$2$ engine into two prongs. Realize $B^{(m)}_j:=[m+1:j]$ inside $\A_{m+1}$ by deleting all but $j$ of the level-one subtrees of type $[m:\omega]$; then $B^{(m)}_j\not\cong\A_{m+1}$, the number of $[m:\omega]$-subtrees being an isomorphism invariant.

\begin{conjecture}[Substitution lemma, two prongs]\label{conj:subst}
Let $m\geq 1$.
\begin{enumerate}
\item[(a)] \emph{($\Pi$-substitution; generalizes Claim~\ref{claim:pi1transfer} and Lemma~\ref{lem:transfers}(i)).} Every $\Pin{2m-1}$ formula true in $\A_{m+1}$ of a tuple $(r,\bar w)$ --- $r$ a fresh root of a $[m:\omega]$-subtree, $\bar w$ arbitrary witnesses, possibly inside deleted $[m:\omega]$-subtrees --- remains true after relocating $r$ to a fresh root of a $[m:k]$-subtree, for all sufficiently large finite $k$, with $\bar w$ carried to matched witnesses.
\item[(b)] \emph{($\Sigma$-richness; generalizes Lemma~\ref{lem:transfers}(ii)).} For every $\theta\in\Pin{2m}$ and every $\bar a$ with $\A_{m+1}\models\theta(\bar a)$: once $j$ exceeds the number of $[m:\omega]$-roots named by $\bar a$ (so $\bar a\subseteq B^{(m)}_j$), one also has $B^{(m)}_j\models\theta(\bar a)$, obtained by re-hosting the inner $\Pin{2m-2}$ witnesses of $\theta$ inside $B^{(m)}_j$.
\end{enumerate}
Consequently, peeling $\Pin{2m+1}=\bigwedge\forall\bar y\,(\Sin{2m})$ with $\Sin{2m}=\bigvee\exists\bar w\,(\Pin{2m-1})$, prong~(a) gives that every $\Pin{2m+1}$ sentence true in $\A_{m+1}$ holds in \emph{every} $B^{(m)}_j$; and peeling $\Sin{2m+1}=\bigvee\exists\bar y\,(\Pin{2m})$, prong~(b) gives that every $\Sin{2m+1}$ sentence true in $\A_{m+1}$ holds in $B^{(m)}_j$ for all sufficiently large $j$. Against the trichotomy of Corollary~\ref{cor:dichotomy}: (a) excludes a $\Pin{2m+1}$ Scott sentence (re-proving Theorem~\ref{thm:SR}); (b) excludes a $\Sin{2m+1}$ one; and (a) together with (b) excludes a $\dSin{2m+1}=\sigma\wedge\pi$ one, since $\pi$ survives in every $B^{(m)}_j$ by (a) and $\sigma$ in all large $B^{(m)}_j$ by (b), so $\sigma\wedge\pi$ holds in some $B^{(m)}_j\not\cong\A_{m+1}$. Hence $\SSC(\A_{m+1})=\Pin{2m+2}$ by the rank-$2$ method --- re-proving Theorem~\ref{thm:main-general} without computability, exactly as at rank~$2$ (Theorem~\ref{thm:rank2}).
\end{conjecture}

The matrix-level asymmetry --- $\Pin{2m-1}$ in (a) versus $\Pin{2m-2}$ in (b) --- is precisely the rank-$2$ phenomenon. At $m=1$, prong~(a) is the $\Pin{1}$ transfer of Claim~\ref{claim:pi1transfer} (used again, in the relocation direction, inside Lemma~\ref{lem:transfers}(i)), while prong~(b) bottoms out at $\Pin{0}$, i.e.\ at quantifier-free matrices --- which is why the proof of Lemma~\ref{lem:transfers}(ii) never invokes Claim~\ref{claim:pi1transfer} but applies the Extension Lemma~\ref{lem:ext} directly to its $\Sin{1}$ conjuncts. The expected proof for $m\geq 2$ is an induction on $m$ run jointly with the definability ladder for the types --- the inductive claim that ``$T[u]\cong[k:j]$'' is $\dSin{2k-1}$- and ``$T[u]\cong[k:\omega]$'' is $\Pin{2k}$-definable, each level of rank costing exactly two quantifier levels (the syntactic shadow of \cite[Lem.~2.1]{Mah19}) --- with the Extension Lemma~\ref{lem:ext} replaced by its rank-$(m+1)$ analogue. The asymmetric back-and-forth toolkit for trees of Harrison-Trainor and Kim \cite[\S\S2--3]{HTK26}, in particular their localization of existential witnesses to subtrees, is well suited to this induction, and we expect their machinery to shorten the bookkeeping considerably. The same induction would localize the non-$\Sin{2m}$-definable orbit of Theorem~\ref{thm:SR} at $O^{m}_\omega$, giving the general case the same dual (computability-free) proof that rank~$2$ enjoys.

\begin{question}\label{q:table}
Which Scott sentence complexities are realized by trees of finite rank? The trees $\A_{m+1}$ pin the even levels $\Pin{2m+2}$ of the $\Pi$-column --- proved at rank $2$ (Theorem~\ref{thm:rank2}, with a computable witness by Proposition~\ref{prop:computable-ss}) and proved in general (Theorem~\ref{thm:main-general}). For the remaining levels, $[m+1:0]$ is the natural candidate for the odd $\Pi$-level $\Pin{2m+1}$, while the trees $[m+1:j]$ with finite $j\geq 1$ are the natural candidates for the $\dSin{2m+1}$-level: by Remark~\ref{rem:localize} the $\omega$-orbit of each $[2:j]$ ($j\geq 1$) is already not $\Sin{2}$-definable, and the witness ``there exist exactly $j$ subtrees of type $[m:\omega]$'' has the characteristic $\mathrm{d}$-shape (a $\Sin{2m+1}$ ``at least $j$'' conjoined with a $\Pin{2m+1}$ ``at most $j$''). At rank $2$ one checks directly that $\SSC([2:0])=\Pin{3}$ and $\SSC([2:j])=\dSin{3}$ for $1\leq j<\omega$, alongside $\SSC([2:\omega])=\Pin{4}$ from Theorem~\ref{thm:rank2}; in particular this family realizes only $\Pi$- and $\mathrm{d}$-$\Sigma$-complexities. Which trees, if any, realize the $\Sigma$-column, and which complexities are excluded --- in analogy with the impossibility of $\Sin{3}$ for linear orders \cite{GR23} --- remain open. A complete table for bounded rank, in analogy with \cite[Table~2]{GR23} for linear orders, is the goal.
\end{question}

\begin{question}\label{q:gaps}
Harrison-Trainor and Kim \cite{HTK26} prove that every $\Pi_\alpha$-axiomatizable class of trees has a member of Scott rank at most $\alpha+2$, with a lower-bound example at the bottom level. Sharpness for general $\alpha$ is open. The calibrated ambiguity of the types --- a $\Pin{2}\setminus\Sin{2}$ orbit at rank~$2$ (Theorem~\ref{thm:notSigma2}), and at each rank an orbit that Theorem~\ref{thm:SR} shows is not $\Sin{2m}$-definable and that the definability ladder of Conjecture~\ref{conj:subst} would place in $\Pin{2m}$ --- is the natural raw material for lower bounds, but a cautionary computation from Section~\ref{sec:ranktwo} shows that single-level coding does not suffice: the $\Pin{3}$ theory ``rank $\leq 2$ and every finite degree occurs infinitely often'' has the models $[2:0],[2:1],\dots,[2:\omega]$, of minimum Scott rank $2$ (attained at $[2:0]$), and Lemma~\ref{lem:transfers}(i) shows no $\Pin{3}$ strengthening can exclude the low model. A sharpness example must make the type-ambiguity survive at \emph{every} level of every model --- precisely the structure that the induction of Conjecture~\ref{conj:subst} would quantify. Can the representation machinery of \cite[Lem.~3.1]{Mah19} be converted into $\Pi_\alpha$ theories of trees all of whose models have Scott rank close to $\alpha+2$?
\end{question}

\begin{question}\label{q:infinite}
Extend the analysis to infinite ranks. For infinite $\alpha$, \cite{Mah19} gives strong degree of categoricity $\mathbf{0}^{(2\alpha+1)}$ at rank $\alpha+1$ and $\mathbf{0}^{(\alpha)}$ at limit rank $\alpha$, with the familiar finite/infinite shift; the boldface translation will inherit this shift, and the rank-$(\gamma+2)$ case ($\gamma$ limit) requires computing the back-and-forth structure of the trees $B(y,s)$ of \cite[Lem.~3.6]{Mah19} from scratch. The limit-rank cases should be compared with the limit analysis of \cite{GR23}.
\end{question}

\begin{question}\label{q:general-principle}
The relativize-and-transfer method of Sections~\ref{sec:transfer}--\ref{sec:relhard}
converts any degree-of-categoricity \emph{lower} bound that is uniform in an oracle,
and whose witnessing isomorphism type is oracle-independent, into
non-$\Sin{n}$-definability of an orbit; pinning the exact Scott sentence complexity
then requires, in addition, a matching boldface $\mathbf{\Delta}^0_n$-categoricity
upper bound. Which known degree-of-categoricity constructions meet the hypotheses,
and what complexities result? The structures of Csima--Franklin--Shore \cite{CFS13},
realizing $\mathbf{0}^{(\alpha)}$ for computable $\alpha$, are natural candidates,
though their characteristic cases are infinite-level and so meet
Proposition~\ref{prop:transfer} only after the finite/infinite shift of
Question~\ref{q:infinite} is resolved. In a different direction,
Lempp--McCoy--Miller--Solomon \cite{LMMS05} classify the \emph{computably} categorical
trees of finite height; combined with the degrees of categoricity from \cite{Mah19},
this raises the question of which Scott sentence complexities are realized across the
whole class of finite-height trees, not only by the maximal-degree members
$\A_{m+1}$ analyzed here.
\end{question}

\subsection*{Statement of contributions}
The authors are listed alphabetically by surname. The construction relativized here, together with the transfer principle (Section~\ref{sec:transfer}), the relativized coding (Section~\ref{sec:relhard}), the computation of $\SR(\A_{m+1})$ and the trichotomy (Section~\ref{sec:SR}), and the complete rank-$2$ analysis (Section~\ref{sec:ranktwo}), are due to M.~A.~Mahmoud, building on his earlier work \cite{Mah19}. The parameter-reservation argument of Section~\ref{sec:general} --- the Reservation Lemma~\ref{lem:reserve}, the parameterized coding Theorem~\ref{thm:paramcoding}, and the deduction $\pSR(\A_{m+1})=2m+1\Rightarrow\SSC(\A_{m+1})=\Pin{2m+2}$ (Theorem~\ref{thm:main-general}) --- is due to M.~Mirabi. The final exposition is the work of both authors.

\subsection*{Acknowledgement}
The finite-rank tree construction relativized here is from M.~A.~Mahmoud's earlier
work \cite{Mah19}; the boldface and Scott-theoretic analysis of
Sections~\ref{sec:transfer}--\ref{sec:ranktwo} was developed by him with the
assistance of AI language models, used for literature exploration, for drafting and
refining exposition and proofs, and as an adversarial check on the arguments. The
tools used were Google's Gemini Pro and Anthropic's Claude (the Fable~5 model during its
brief period of availability, and Opus version~4.8). The parameter-reservation
argument of Section~\ref{sec:general} was developed by M.~Mirabi without the use of
any AI tool. All definitions, theorems, and proofs were verified by the authors, who
are jointly responsible for the contents and for any errors; a full description of the
AI use has been provided to the editors.

\end{document}